\documentclass[12pt]{article}
\usepackage[usenames]{color}
\usepackage{colortbl}
\usepackage{tikz-cd}
\usepackage{xcolor}
\usepackage[bbgreekl]{mathbbol}

\usepackage{amscd,amsmath, amssymb, fancyhdr, mathbbol,  url, euscript, amssymb}
\usepackage[backref=page]{hyperref}
\usepackage{mathtools}

\hypersetup{
	colorlinks   = true,
	citecolor    = magenta,
	linkcolor    =blue,
	urlcolor     =magenta	
}

\numberwithin{equation}{section}
\newcommand{\version}{}

\def\eqref#1{(\ref{#1})}

\def\1{\sqrt{-1}\:}

\newcommand{\cntrct}                % contraction with a vector field
{\hspace{2pt}\raisebox{1pt}{\text{$\lrcorner$}}\hspace{2pt}}

\renewcommand{\tilde}{\widetilde}
\renewcommand{\bar}{\overline}
\renewcommand{\phi}{\varphi}
\renewcommand{\epsilon}{\varepsilon}
\renewcommand{\geq}{\geqslant}
\renewcommand{\leq}{\leqslant}

\newcommand{\be}{\begin{equation}}
\newcommand{\bea}{\begin{eqnarray}}
\newcommand{\eea}{\end{eqnarray}} 
\newcommand{\ee}{\end{equation}}

\newcommand{\Tr}{\operatorname{Tr}}

\newcommand{\vp}{\varphi}
\newcommand{\cI}{\mathcal{I}}
\renewcommand{\i}{\infty}
\newcommand{\cH}{\mathcal{H}}
\newcommand{\ve}{\varepsilon}
\newcommand{\p}{\partial}
\renewcommand{\d}{\delta}
\newcommand{\D}{\Delta}
\newcommand{\na}{\nabla}
\newcommand{\cO}{\mathcal{O}}

\renewcommand{\b}{\beta}

\renewcommand{\o}{\omega}
\renewcommand{\a}{\alpha}
\newcounter{Mycounter}[section]
\newcounter{lemma}[section]
\newcounter{claim}[section]
\newcounter{sublemma}[section]
\newcounter{corollary}[section]
\newcounter{theorem}[section]
\newcounter{conjecture}[section]
\newcounter{proposition}[section]
\newcounter{definition}[section]
\newcounter{example}[section]
\newcounter{remark}[section]
\newcounter{problem}[section]
\newcounter{question}[section]
\makeatletter

\@addtoreset{equation}{section} \@addtoreset{footnote}{section}
\makeatother

\begin{document}
\title{Convergence of the inverse Monge-Ampere flow and Nadel multiplier ideal sheaves.\\ {\small Nikita Klemyatin}}
\date{}
\maketitle

\begin{abstract}
	We generalize the inverse Monge-Ampere flow, which was introduced in \cite{CHT17} by Collins, Hisamoto and Takahashi, and provide conditions that guarantee the convergence of the flow without a priori assumption that $X$ has a K\"ahler-Einstein metric. We also show that if the underlying manifold does not admit K\"ahler-Einstein metric, then the flow develops Nadel multiplier ideal sheaves. In addition, we establish the linear lower bound for $\inf_X\vp$, and the theorem of Darvas and He for the inverse Monge-Ampere flow.
	{\footnotesize }
\end{abstract}

%\tableofcontents

\section{Introduction}

\hfill

The search for K\"ahler-Einstein metrics in K\"ahler geometry goes back to the famous work of Calabi \cite{Calabi}, and Yau's solution to Calabi conjecture \cite{Yau}. It is well-known that on Fano manifolds there are certain obstructions to the existence of K\"ahler-Einstein metrics.

Suppose that $(X,\o)$ is a compact K\"ahler manifold, $\b$ is a closed semipositive $(1,1)$-form, and $[\o] = c_1(X) - [\b]$. Let $\cH = \{\vp \in C^\i ~|~ \o + \sqrt{-1}\p \bar{\p} \vp > 0 \}$. For a given metric $\o_\vp$ within the cohomology class $[\o]$ we have the following identity:
\bea
\textrm{Ric}(\o_\vp) = \o_\vp + \b + \sqrt{-1}\p \bar{\p}\rho,
\eea
where $\rho = \rho_\vp$ is a function (usually called the Ricci potential), that is defined up to a constant. The identity above is equivalent to a Monge-Ampere type identity:
\bea
e^\rho \o^n_\vp = e^{\rho_0 - \vp + c_\o}\o^n_0.
\eea

Here $c_\o$ is a normalizing constant.

One could try to find $\vp$, such that the corresponding $\rho$ equals zero. In other words, one may try to find a function $\vp$, such that
\bea
\textrm{Ric}(\o_\vp) = \o_\vp + \b.
\eea

Such metric $\o_\vp$ (if it exists) is called the {\it twisted K\"ahler-Einstein metric}. They satisfy the following Monge-Ampere equation:
\bea
\o^n_\vp = e^{\rho_0 - \vp + c_\o}\o^n_0.
\eea

It is well-known that this equation does not always admit a smooth solution.

There is a great interest in searching for such metrics, with $\b$ being either a semipositive form (see \cite{CS12, DatarSze}), or a positive current (see \cite{BBEGZ11, BBJ15, GP16, Z2021, ZZ21, RossSze} and the references therein). 

We consider the inverse Monge-Ampere flow, introduced in \cite{CHT17} for the case $\b = 0$. The inverse Monge-Ampere flow is a gradient flow for the functional $\mathcal{F}(\vp)$ (see \ref{FFunctional} below for the definition). On the level of potentials, it can be written as follows:
\bea
\dot{\vp} = 1 - e^\rho = 1 - {\o^n_0 \over \o^n_\vp}e^{-\vp + \rho + c(t)}.
\eea

Here $\rho$ is the Ricci potential of the evolving metric $\o_\vp$, such that ${1 \over V}\int_X e^\rho \o^n_\vp = 1$, and $c(t) = - \log \big({1 \over V}\int_X e^{\rho_0 -\vp} \o^n_0\big)$ is a normalizing constant. Since the $(1,1)$-form $\beta$ does not affect the equation on the level of potentials, we will call it the inverse Monge-Ampere flow without, regardless of the presence of $\beta$.

It was shown in \cite[Proposition 2.7]{CHT17} that the uniform equivalence of evolving metrics implies the uniform estimates for the solution. In other words, if there is a constant $A>0$, such that $A^{-1}\o_0 \leq \o_\vp \leq A\o_0$, then all the derivatives of $\vp$ are Holder continious.  This condition implies a sequential convergence, but it is a way too strong assumption, which is hard to establish. We want to find milder conditions, that would guarantee the convergence, at least in the case when $H^0(X, TX) = 0$.

In order to find such conditions, in section \ref{SectionAlphaInvariant} we study the relations between $\a$-invariant and the inverse Monge-Ampere flow. In \ref{SupBoundsForAlphaInv} we show, that if $\a$-invariant is bigger than ${n \over n+1 }$, then one has the upper bound for $\sup_X \vp$. Hence, one could think that this bound should imply the convergence of the flow. The following theorem says that it is true. 

\hfill

\begin{theorem}\label{TheoremOne}
	Assume that $(X,\o_0)$ is a compact K\"ahler manifold. Let $\b \geq 0$ be a semipositive $(1,1)$-form, such that $c_1(X) = [\o_0] + [\b]$. Assume further, that $X$ does not have holomorphic vector fields. 
	Then the following estimates along the inverse Monge-Ampere flow are equivalent and imply the convergence of the flow:
	\begin{itemize}
		\item[(1)] $\sup_X \vp$ is uniformly bounded from above;
		
		\item[(2)] The average ${1 \over V}\int_X \vp \o^n_0$ is uniformly bounded from above;
		
		\item[(3)] $J(\vp)$ and $d_1(0,\vp)$ are uniformly bounded from above;
		
		\item[(4)] $I(\vp)$ is uniformly bounded;
		
		\item[(5)] For some $p>1$ the integral ${1\over V}\int_X e^{-p\vp}\o^n_0$ is uniformly bounded.
	\end{itemize}
	
	The following bounds imply the convergence of the inverse Monge-Ampere flow.
	\begin{itemize}
		\item[(6)] The alpha-invariant $\a(X, \o_0)$ is bigger than ${n \over n+1}$;
		
		\item[(7)] $\textrm{osc}(\vp)$ is uniformly bounded along the flow;
		
		\item[(8)] $\inf_X\vp$ is uniformly bounded below along the flow;
		
		\item[(9)] For some $p > 1$ the distance $d_p(0, \vp)$ is uniformly bounded from above along the flow.
	\end{itemize}
	In particular, if $(X,\o_0)$ does not admit a twisted K\"ahler-Einstein metric, then both $||\vp||_{C^0}$ and ${1 \over V}\int_X \vp \o^n_0$ are unbounded.
\end{theorem}

\hfill

These bounds are known for the case of K\"ahler Ricci flow, see \cite{CW1, PSS2006, PSSW2007,Rub2009}. Again, as in \cite{PSS2006}, the bound on the number $p$ in (5) is sharp, and $p=1$ does not imply the convergence of the flow. However, the upper bound in the borderline case $p=1$ implies semistability of $X$.

It is worth mentioning that in \cite{Rub2009} the similar theorem about the multiplier ideal sheaves was proved using the Perelman estimates and the uniform Sobolev inequality along the K\"ahler-Ricci flow. For the twisted Monge-Ampere flow, there is no known uniform Sobolev inequality. Moreover, there is no analog of Perelman's estimates. It was shown in \cite[Section 6]{CHT17}, that the Ricci potential may be unbounded along the flow if the underlying Fano manifold is unstable. Thus, we cannot rely on the methods, that worked in the case of the K\"ahler-Ricci flow. Instead, we use the pluripotential theory to obtain the above theorem.

As a result, the claims for $\inf_X \vp$ and
for $\sup_X \vp$ in \ref{TheoremOne} are obtained in very different ways. In particular, the proof of this theorem does not show any direct connection (like Harnack inequality in the case of K\"ahler-Ricci flow) between $\inf_X \vp$ and $\sup_X \vp$. Below, in \ref{TheoremThree}, we will address this particular issue. 

Next, as an application of \ref{TheoremOne} and the results from section \ref{SectionAlphaInvariant}, we show that if manifold does not admit a twisted K\"ahler-Einstein metric, then the inverse Monge-Ampere flow produces a Nadel multiplier ideal sheaf.

\hfill

\begin{theorem}\label{TheoremTwo}
	Let $(X,\o_0)$ be a compact K\"ahler manifold, that satisfies all the assumptions of \ref{TheoremOne}, and does not admit a twisted K\"ahler-Einstein metric. Let $\a >{n \over n+1}$. Then there is a sequence of times $\{t_j\}$, and a sequence of solutions $\{\vp_{t_j}\}$ to the inverse Monge-Ampere flow, such that $\psi_j:= \vp_j - {1 \over V}\int_X \vp_j \o^n_{\vp_j}$ converge in $L^1(X, \o^n_0)$ to $\psi_\i$, and $\cI(\a\psi_\i)$ is a proper multiplier ideal sheaf, and
	
	\bea
	H^q(X,-\lfloor \a \rfloor K_X \otimes \cI(\a\psi_\i)) = 0, ~~ \forall q\geq 1.
	\eea
\end{theorem}

\hfill

Thus, \ref{TheoremOne} and \ref{TheoremTwo} also recover the results of \cite{PSS2006} and \cite{Rub2009}. 

Finally, we address the question, asked by \cite[Section 5.2]{Xia2019} about the lower linear bound on $\inf \vp$. Originally, the estimate for $\inf_X \vp$ in finite time was obtained in \cite{CHT17} by using the Kolodziej $L^\i$ estimate. However, this estimate is implicit, and it does not give any actual qualitative bounds. We answer this question under the additional assumption that $\o^n_\vp$ is exponentially bounded along the flow.

\hfill

\begin{theorem}\label{TheoremThree}
	Let $\vp$ be a solution to the Inverse Monge-Ampere flow.
	Then there is a constant $M > 0$, such that 
	\bea
	||\vp||_{C^0} \leq M(t + 1)
	\eea
	along the inverse Monge-Ampere flow. 
\end{theorem}

\hfill

The converse is also true, and it follows from the known bounds for the inverse Monge-Ampere flow.

Using \ref{TheoremThree}, and the results from section \ref{SectionAlphaInvariant}, we show, how to establish results from \cite{DarvasHe} in our setting.

\hfill

\begin{theorem}\label{TheoremFour}
	Assume that $\vp_t$ is a diverging trajectory of the inverse Monge-Ampere flow, which satisfies the assumptions of \ref{TheoremThree}. Then there exists a curve $\vp^t$, such that for any $p \geq 1$ it is a nontrivial $d_p$ geodesic ray, weakly asymptotic to $\vp_t$. Moreover, the functional $\mathcal{F}(\vp^t)$ is convex and decreasing along the geodesic ray, and the normalized ray $\vp^t - \sup_X (\vp^t - \vp^0)$ converges to  $\vp^\i$, such that ${1 \over V}\int_X e^{-{n \over n+1}\vp^\i}\o^n_0 = +\i$.
\end{theorem}

\hfill

The organization of the paper is as follows. In section \ref{SectionPrelim} we collect the basic information about the inverse Monge-Ampere flow and the behavior of various functionals along the flow. In \ref{SectionAlphaInvariant} we establish the relationships between the inverse Monge-Ampere flow and the alpha-invariant. In the absence of Perelman estimates for the inverse Monge-Ampere flow, we rely on the monotonicity of the Mabuchi energy $\mathcal{M}(\vp)$. This allows us to prove \ref{SupBoundsForAlphaInv}. In section \ref{SectionConvergence} we prove \ref{TheoremOne}. In section \ref{SectionNadel} we prove \ref{TheoremTwo}. Finally, in section \ref{SectionSuffering} we obtain \ref{TheoremThree} and \ref{TheoremFour}, and discuss its application.

\textbf{Acknoledgements.} The author is grateful to Professor D.H.Phong for his encouragement and support. I would like to thank Yulia Gorginyan and Chuwen Wang for 
pointing out several typos in the earlier versions of this text.

\section{The Inverse Monge-Ampere flow.}\label{SectionPrelim}
%\subsection{Basic estimates along the flow.}\label{BasicPrelim}

\hfill

Consider a compact K\"ahler manifold $(X,\o_0)$, and a closed semipositive form $\b$, such that $[\o_0] + [\b]= c_1(X)$. Let $\cH = \{\vp \in C^\i(X) ~|~ \o_0 + \sqrt{-1}\p \bar{\p} \vp > 0 \}$,and $\rho=\rho_\vp$ is the Ricci potential of the metric $\o_\vp$, i.e. 
\be
\textrm{Ric}(\o_\vp) = \o_\vp + \b + \sqrt{-1}\p \bar{\p} \rho, ~~~~
{1 \over V} \int_X e^{\rho} \o^n_\vp  = 1, ~~~~ V = \int_X \o^n_0.
\ee

We consider the following flow:
\bea\label{TheInverseMA}
\dot{\vp} = 1 - e^{\rho}
\eea

or, writing it in terms of evolution of the metric, 
\bea
\dot{g}_{j \bar k} = -\p_j \p_{\bar k} (e^\rho) 
= - \rho_{j \bar k} e^\rho -  \rho_j \rho_{\bar k} e^\rho
=-e^{\rho}(R_{j \bar k} - g_{j \bar k} - \b_{j \bar k} + \rho_j \rho_{\bar k}). 
\eea

We may rewrite the flow as parabolic MA equation, similar to an equation, introduced by Krylov \cite{Krylov76}:

\bea
\dot{\vp} = 1 - {\o^n_0  \over \o^n_\vp} e^{ -\vp + \rho_0 + c(t)},  ~~~c(t) = -\log \Bigg({1 \over V} \int_X e^{-\vp + \rho_0}\o_0^n \Bigg). 
\eea

Here we used the fact that the solution of the inverse MA flow satisfies the following identity:
\bea\label{RicciPotentialIdentity}
e^\rho \o^n_\vp = e^{-\vp + \rho_0 + c(t)}\o^n_0.
\eea

Below we collect the basic estimates, which are known from \cite{CHT17} in the case $\beta = 0$.

\hfill

\begin{proposition}\label{KnownEstimates}
	Along the inverse Monge-Ampere flow with $\beta \geq 0$ the following estimates hold:
	\begin{enumerate}
		\item $ \vp \leq t + A_1$;
		
		\item $ \rho \geq - t + c(t) + A_2$;
		
		\item $c(0) + t\dot{c}(0) \leq c(t) \leq \sup_X \vp - \inf_X \rho_0$.
	\end{enumerate}
	Here $A_1$ and $A_2$ are uniform constants.
\end{proposition}

\hfill

\begin{proof}
	
	\begin{enumerate}
		\item Follows from the equation of the flow \ref{TheInverseMA};
		
		\item The proof is the same as in \cite[Lemma 4.2]{CHT17};
		
		\item The first inequality follows from the convexity of $\mathcal{F}(\vp)$. It is proved in \ref{FConvex} below. The second inequality follows from the definition of $c(t)$ and the Jensen inequality:
		\bea
		c(t) = -\log \Bigg({1 \over V} \int_X e^{-\vp + \rho_0}\o_0^n \Bigg)
		\leq
		\int_X (\vp - \rho_0)\o_0^n
		\leq 
		\sup_X \vp - \inf_X \rho_0.
		\eea
	\end{enumerate}
\end{proof}

\hfill

Now we need to recall some functionals on $\cH$. First of all, we have the $I$-functional, which is given by 

\bea
I(\vp) = {1 \over V}\int_X \vp\big(\o^n_0 - \o^n_\vp\big).
\eea

The $J$-functional is defined to be

\bea
J(\vp) = {1 \over V} \sum_{j=0}^{n-1} {n-j \over n+1}\int_X \sqrt{-1}\p \vp \wedge \bar{\p}\vp \wedge \o^j_\vp \wedge \o^{n-j-1}_0.
\eea

There is an inequality between these two functionals: ${1 \over n}J \leq I-J \leq nJ$.

The functional $\mathcal{F}(\vp)$ and the energy $\textrm{E}(\vp)$ are given by the following:

\begin{equation}\label{FFunctional}
	\begin{split}
		\mathcal{F}(\vp) & = -\log \Bigg({1 \over V} \int_X e^{-\vp + \rho_0}\o^n_0\Bigg) - \textrm{E}(\vp) = c(t) - \textrm{E}(\vp), 
		\\
		\textrm{E}(\vp) &=  {1 \over V}\int_X \vp \o^n_0 - J(\vp),
		\\
		\d \textrm{E}(\vp) &= {1 \over V} \int_X \d \vp \o^n_\vp.
	\end{split}
\end{equation}

We note that the inverse Monge-Ampere flow is the gradient functional for $\mathcal{F}(\vp)$ (see formula \ref{VariationOfFunctionals} below).  

Finally, the Mabuchi functional is defined by

\be\label{MabuchiFunctional}
\mathcal{M}(\vp):= {1 \over V}\int_X \log {\o^n_\vp \over \o^n_0}\o^n_\vp - {1 \over V} \int_X \rho_0 \o^n_\vp - (I-J)(\vp).
\ee 

It is worth to mention the relation between Mabuchi functional and $\mathcal{F}(\vp)$ along the inverse Monge-Ampere flow.

\hfill

\begin{proposition}\label{RelationFMabuchi}
	Along the inverse Monge-Ampere flow we have 
	\begin{equation*}
		\mathcal{M}(\vp) = \mathcal{F}(\vp) - {1 \over V} \int_X \rho \o^n_\vp.
	\end{equation*}
	In particular, $\mathcal{M}(\vp) \geq \mathcal{F}(\vp)$ along the flow. 
\end{proposition}

\hfill

\begin{proof}
	\bea
	\mathcal{M}(\vp)
	=
	{1 \over V}\int_X (-\rho - \vp + \rho_0 + c(t))\o^n_\vp - {1 \over V} \int_X \rho_0 \o^n_\vp - (I-J)(\vp)
	\\
	=
	J(\vp) - {1 \over V}\int_X \vp \o^n_0 + c(t) - {1 \over V} \int_X \rho \o^n_\vp
	=
	\mathcal{F}(\vp) - {1 \over V} \int_X \rho \o^n_\vp.
	\eea
	
	By the Jensen inequality, we know that 
	\bea
	{1 \over V} \int_X \rho \o^n_\vp
	\leq 
	\log \Bigg({1 \over V} \int_X e^\rho \o^n_\vp\Bigg)
	\leq
	0.
	\eea
\end{proof}

\hfill

The variation of $\mathcal{F}(\vp)$ and $\mathcal{M}(\vp)$ are given by the following expressions:

\begin{equation}\label{VariationOfFunctionals}
\begin{split}
	\d \mathcal{F}(\vp) & = {1 \over V} \int_X \d \vp (e^\rho - 1)\o^n_\vp,
	\\
	\d \mathcal{M}(\vp) &= -{1 \over V} \int_X \d \vp (R - n - \textrm{Tr}_{\o_\vp}\b)\o^n_\vp.
\end{split}
\end{equation}

The following lemma describes the behaviour of these functionals along the flow.

\hfill

\begin{lemma}\label{FAndMabuchi}
	
	Along the inverse Monge-Ampere flow we have the following identities
	\bea
	\dot{\mathcal{F}}(\vp) = -{1 \over V} \int_X (e^\rho - 1)^2 \o^n_\vp,
	\\
	\dot{\mathcal{M}}(\vp) = - {1 \over V}\int_X |\na \rho|^2 e^\rho \o^n_\vp.
	\\
	\dot{\textrm{E}}(\vp) = 0.~~~~~~~~~~~
	\eea
	
	In particular, $\dot{c}(t) \leq 0$ along the flow.
\end{lemma} 

\hfill

\begin{proof}
	
	The statement for $\mathcal{F}(\vp)$ functional follows from the variation formula and the flow equation $\dot{\vp} = 1 - e^\rho$. For $\mathcal{M}(\vp)$ we have the following identity along the flow:
	\begin{equation}
		\dot{\mathcal{M}}(\vp) 
		=
		- {1 \over V} \int_X (1 - e^\rho) \D\rho\o^n_\vp
		=
		-{1 \over V} \int_X |\na \rho|^2 e^\rho\o^n_\vp
		\leq 0.
	\end{equation}

	Finally, the statement for $\dot{\textrm{E}}(\vp)$ follows from the normalization condition for $\rho$ along the flow.
\end{proof}

\hfill

\begin{remark}
	
	The content of this lemma was proved in various statements in \cite{CHT17}. However, our proof is simpler for $\mathcal{M}(\vp)$ than \cite[Lemma 4.7]{CHT17}.
\end{remark}

\hfill

In \cite[Proposition 2.4 and Corollary 2.5]{CHT17} authors also compute the second derivative of $\mathcal{F}(\vp)$ along the inverse Monge-Ampere flow:

\bea\label{FSecondDer}
\ddot{\mathcal{F}}(\vp) = {2 \over V} \int_X \Big(|\na \tilde{f}|^2 - |\tilde{f}|^2\Big)d\mu ,
~\\
\tilde{f} := e^\rho -1 - {1 \over V} \int_X ( e^\rho -1) d\mu,
~\\
d\mu = e^\rho \o^n_\vp.~~~~~~~~~~~
\eea

We need the following lemma.

\hfill

\begin{lemma}\label{FConvex}
	
	If $\b \geq 0$, then $\mathcal{F}$ is convex along the inverse Monge-Ampere flow.
\end{lemma}

\hfill

\begin{proof}
	
	The statement would easily follow the fact that the operator $-L_\rho =\na^{\bar m}\na_{\bar m} f + \na^{\bar m}\rho \na_{\bar m}$. The proof for $\b = 0$ is given in \cite[Lemma 2]{PSSW2007}, and it applies here as well. For reader's convenience, we briefly recall the main steps here. 
	
	If $\b \geq 0$, then the application of $\na_{\bar m}$ to the equation $\na^{\bar m}\na_{\bar m} f + \na^{\bar m}\rho \na_{\bar m}f= -\nu f$, and integration against $\na^{\bar m}f d\mu$ gives us the following identity
	
	\bea
	(\nu-1)\int_X |\na f|^2 d\mu 
	= 
	\int_X|\bar{\na}\bar{\na} f|^2 + \b^{j \bar k}\na_jf \na_{\bar k}.
	f d\mu\eea
	
	The claim about the spectrum follows from the nonnegativity of $\b$.
\end{proof}

\hfill

Now we state the theorem about the behavior of the twisted inverse MA flow.

\hfill

\begin{theorem}
	Let $(X,\o_0)$ be a compact K\"ahler manifold with no holomorphic vector fields, and $c_1(X) = [\o_0] + [\b]$ for a nonnegative $(1,1)$-form $\b$. Then the inverse Monge-Ampere flow exists for all times, and if there is a twisted K\"ahler-Einstein metric, then the flow converges to it. %modulo the action of $\Aut_0(X,\b) = \{ g \in \Aut_0(X) ~|~ g^*\b = \b \}$.
\end{theorem}

\hfill

\begin{proof}
	
	Since the flow looks the same for both twisted and non-twisted cases, $\mathcal{F}(\vp)$ is convex, the arguments from \cite[Theorem 4.6 and Theorem 4.11]{CHT17} together with \cite{BerBernd,BDL,Brn1} carry over verbatim to the twisted case.
\end{proof}

\hfill

\begin{remark}
	
	In fact, one do not really need the presence of the KE metric. As it was pointed out in \cite{CHT17}, the coercivity of $\mathcal{F}(\vp)$ or $\mathcal{M}(\vp)$ is enough for existence of the twisted KE metric. %, and from this one can deduce the convergence of the flow modulo $\Aut_0(X,\b)$. 
\end{remark}

\hfill

\section{The inverse Monge-Ampere flow and the alpha-invariant.}\label{SectionAlphaInvariant}

In this section, we show that the alpha-invariant inequality holds along the inverse MA flow. It is similar to the inequality along the K\"ahler-Ricci flow. Let

\bea
\a(X, \o_0):= \sup \Biggl\{\a >0 ~\vline~ \int_X e^{-\a\vp}\o^n_0 < +\i, ~ \forall\vp \in \cH, ~\sup_X \vp = 0 \Biggr\}.
\eea

This is a well-known $\a$-invariant, which was introduced by Tian \cite{TianAlpha} in the case of Fano manifolds. It was shown in \cite{TianAlpha}, that a Fano manifold $X$ without holomorphic vector fields, which satisfies $\a(X,c_1(X))>{n \over n+1}$, admits a K\"ahler-Einstein metric. Hence, by the results from \cite{CHT17}, one has the convergence of the inverse Monge-Ampere flow. However, this is not a direct way to obtain the convergence of the flow, and one would like to find a more direct approach, similarly to \cite{CW1,Rub2009}, where the relations between the $\a$-invariant and the convergence of the K\"ahler-Ricci flow were established directly. 

Before we proceed, we need a lemma:

\begin{lemma}\label{SupBounds}
	Along the inverse MA flow we have the following inequality:
	\bea\label{SupBoundTwo}
	-{1 \over V} \int_X \vp \o^n_\vp \leq n\sup_X\vp + C.
	\eea
\end{lemma}

\begin{proof}
	Note that for a negative PSH-function $\psi = \vp - \sup_X \vp$ we have the following inequality:
	
	\bea
	&-{1 \over V}\int_X \psi \o^n_\psi \leq -(n+1)\textrm{E}(\psi) = -(n+1)\textrm{E}(\vp) + (n+1)\sup_X\vp,
	\eea
	
	and $\textrm{E}(\vp)$ is constant along the flow. Note also that $\o_\psi = \o_\vp$, and
	
	\bea
	-{1 \over V}\int_X \psi \o^n_\psi = -{1 \over V}\int_X \vp \o^n_\vp + \sup_X \vp
	\eea
\end{proof}

\hfill

Now we state the alpha-invariant inequality along the flow:

\hfill

\begin{proposition}\label{SupBoundsForAlphaInv}
	For any $\a > 0$ the following inequality holds along the inverse Monge-Ampere flow:
	\bea
	((n+1)\a-n)\sup_X \vp 
	\leq
	\log \Bigg({1 \over V}\int_X e^{-\a(\vp - \sup_X \vp)} \o^n_0\Bigg) + C.
	\eea
	If $\a(X,\o_0) > {n \over n+1}$, then there is a uniform constant $C_1>0$, such that $\sup_X \vp \leq C_1$ along the inverse Monge-Ampere flow. 

\end{proposition}

\hfill

\begin{proof}
  	By the Jensen inequality and \ref{RicciPotentialIdentity}, we have 
	
	\bea
	{1 \over V} \int_X e^{-\a(\vp - \sup_X\vp)}\o^n_0
	=
	{1 \over V} \int_X e^{-\a(\vp - \sup_X\vp)}\cdot e^{\rho + \vp - \rho_0 - c(t)} \o^n_\vp
	\\
	\geq
	\exp\Bigg({1 - \a \over V} \int_X \vp \o^n_\vp + \a \sup_X\vp + {1 \over V} \int_X (\rho - \rho_0 - c(t)) \o^n_\vp\Bigg).
	\eea 
	Taking the logarithm of the both sides and using \ref{SupBounds}, we obtain the following inequality:
	
	\begin{align}\label{PreAlphaEstimate0}
		((n+1)\a-n)\sup_X\vp   \leq  {1 \over V}\int_X \Big(-\rho + \rho_0  + c(t)\Big) \o^n_\vp + 
		\log \Bigg({1 \over V}\int_X e^{-\a(\vp - \sup_X \vp)} \o^n_0\Bigg).
	\end{align}

	By \ref{RelationFMabuchi} and \ref{FAndMabuchi}, we have the following
	\bea
	\mathcal{M}(\vp(0)) \geq \mathcal{M}(\vp) 
	= \mathcal{F}(\vp) - {1 \over V}\int_X \rho \o^n_\vp =
	c(t)-{1 \over V}\int_X \rho \o^n_\vp + C.
	\eea
	
	Hence, the first term on the right hand side is bounded above, and the first statement follows. 
	Now, if $\a > {n \over n+1}$, then $((n+1)\a-n) > 0$, and the right hand side in \ref{PreAlphaEstimate0} is bounded. Hence, the second statement follows.
\end{proof}

\hfill

\section{The convergence of the inverse Monge-Ampere flow.}\label{SectionConvergence}
In this section we prove \ref{TheoremOne}. First, we show that the convergence of the flow is guaranteed if $\sup_X \vp$ is bounded. 

\hfill

\begin{proposition}\label{UniformSupremumBounds}
	Let $(X, \o_0)$ be a compact K\"ahler manifold with $H^0(X,T^{1,0} X) = 0$, and $[\o_0] +{\b} = c_1(X)$ for a closed nonnegative form $\b$. Suppose $\vp = \vp(t)$ is a solution to the inverse Monge-Ampere flow. Then following statements are equivalent:
	\begin{itemize}
		\item[1] The inverse Monge-Ampere flow converges and $X$ has a twisted K\"ahler-Einstein metric;
		
		\item[2] There is a uniform constant $C>0$, such that $\sup_X \vp \leq C$;
		
		\item[3] There is a uniform constant $C>0$, such that the average ${1 \over V}\int_X \vp \o^n_0$ is bounded from above.
	\end{itemize}
\end{proposition}

\hfill

\begin{proof}
	The implication $(1) \implies (3)$ follows from the argument \cite{CHT17}. Indeed, if a twisted KE metric exists, then by \cite{PSSW2007}, the coercivity of the Mabuchi functional implies the upper bound on ${1 \over V}\int_X \vp \o^n_0$. The equivalence $(2) \iff (3)$ follows from the standard inequality 
	
	\bea
	{1 \over V}\int_X \vp \o^n_0 
	\leq
	\sup_X \vp
	\leq
	{1 \over V}\int_X \vp \o^n_0 + C(X, \o_0).
	\eea
	
	Now we show the implication $(2) \implies (1)$. The proof is a modification of the arguments contained in \cite{CHT17}.
	
	Recall that $\textrm{E}(\vp)$ is constant along the flow. From this fact and the upper bound on $\sup_X \vp$ we infer that there is sequence of times $\{t_j\}$, such that $\vp_j = \vp(t_j)$ converges to $\vp_\i$ in $L^1(X, \o^n_0)$. For $\psi_j:= -\vp_j + \sup_X \vp_j$ we have the following inequality
	
	\bea
	-{1 \over V}\int_X \psi_j \o^n_j \leq C, ~~ \o_j = \o_{\vp_j} = \o_{\psi_j}
	\eea 
	
	for a uniform constant $C$. This implies that $\psi_\i$ has zero Lelong numbers (see \cite[Corollary 1.8 and Corollary 2.7]{GZ07}). Hence $\vp_\i$ also has zero Lelong numbers.  Then Skoda's Theorem \cite[Theorem 8.11]{GZBook} says that
	
	\bea
	{1 \over V}\int_X e^{-p\vp_\i}\o^n_0 < +\i
	\eea
	
	for all $p>0$. Moreover, the effective version of semicontinuity theorem also guarantee that $e^{-p\vp_j}$ converge to $e^{-p\vp_j}$ in $L^1(X, \o^n_0)$. We can conclude that there is a uniform bound for $||e^{-\vp_j}||_{L^1(X, \o^n_0)}$. This, in turns, implies that 
	
	\bea
	c(t) = - \log \Big({1 \over V}\int_X e^{\rho_0-\vp}\o^n_0\Big)
	\eea
	
	is uniformly bounded. Hence there is a constant $B>0$, such that $\mathcal{F}(\vp_j) \geq -B$.
	Now, we have ${1 \over V}\int_X \vp_j \o^n_j \leq \sup_X \vp_j \leq C$, so $J(\vp_j)$ is bounded by a uniform constant, because
	\bea
	J(\vp_j) = {1 \over V}\int_X \vp_j \o^n_j + \textrm{E}(\vp) = {1 \over V}\int_X \vp_j \o^n_j + \textrm{const}.
	\eea
	along the inverse Monge Ampere flow. Hence, $\vp_j \rightarrow \vp_\i$ in the strong topology. Indeed, one can find constants $\ve, C > 0$, independent on $j$, such that the inequality 
	\bea
	\mathcal{M}(\vp_j) \geq \ve J(\vp_j) -C
	\eea
	holds along the sequence $\vp_j$. By \cite[Theorem 4.14]{BBEGZ11}, our sequence is strongly compact. The arguments from \cite[pages 72-73]{CHT17} allow us to show that $\vp_\i$ is smooth and satisfies the following Monge-Ampere equation;
	
	\bea
	{1 \over V} \o^n_{\vp_\i} 
	= 
	{e^{\rho_0 - \vp_\i}\o^n_0 \over \int_X e^{\rho_0 - \vp_\i}\o^n_0},
	\eea
	
	hence, defines a unique twisted K\"ahler-Einstein metric. The convergence of the flow now follows from the uniqueness result due to the generalization of Bando-Mabuchi result to the twisted case due to Berndtsson \cite[Section 6]{Brn1}.	
\end{proof}

%\hfill

%\begin{remark}
	%If $H^0(X, T^{1,0}X) \neq 0$, then the bound on $\sup_X \vp$ still implies the existence of a twisted KE metric. The convergence of the flow then follows from the same argument as in \cite[Theorem 4.11]{CHT17}.
%\end{remark}

\hfill

Now we show that the $L^p$-integrability of $e^{-\vp}$ implies the coercivity of the Mabuchi functional. This result is similar to the result from \cite{PSS2006} for the K\"ahler-Ricci flow.

\hfill

\begin{proposition}\label{lpestimate}
	If for some $p>1$ the integral ${1 \over V} \int_X e^{-p\vp} \o^n_0$ is uniformly bounded along the flow, then the Mabuchi functional is coercive, and the flow converges to the twisted K\"ahler-Einstein metric.
\end{proposition}

\hfill

\begin{remark}
	If the flow converges, then $\sup_X \vp$ is uniformly bounded, and the integral ${1 \over V} \int_X e^{-p\vp} \o^n_0$ is bounded by the same arguments, as in the proof of \ref{UniformSupremumBounds}.
\end{remark}

\hfill

\begin{proof}
	By the standard argument with the Jensen inequality (see for example \cite[Proposition 6.2]{Z2021}), we have
	
	\begin{multline}
		\log \Bigg({1 \over V} \int_X e^{-p\vp} \o^n_0\Bigg) 
		= 
		\log \Bigg({1 \over V} \int_X e^{-p\vp} { \o^n_0 \over \o^n_\vp} \o^n_\vp\Bigg) 
		\geq 
		{-p \over V} \int_X \vp \o^n_\vp - {1 \over V} \int_X \log \Big({\o^n_\vp 	\over \o^n_0}\Big)\o^n_\vp 
		\\
		= 
		pI(\vp) - {p \over V}\int_X \vp \o^n_0 - {1 \over V} \int_X \log \Big({\o^n_\vp \over \o^n_0}\Big)\o^n_\vp 
		= 
		p(I - J)(\vp) - {1 \over V} \int_X \log \Big({\o^n_\vp \over \o^n_0}\Big)\o^n_\vp - C_0,
	\end{multline}
	
	where $C_0 = \textrm{E}(\vp)$.
	
	If  ${1 \over V} \int_X e^{-p\vp} \o^n_0 \leq C$ for $p>1$, then 
	
	\bea
	{1 \over V} \int_X \log \Big({\o^n_\vp \over \o^n_0}\Big)\o^n_\vp
	\geq
	p(I - J)(\vp) - B
	\eea
	for some constant $B$. From the definition of the Mabuchi functional, we see that there are constants $\ve = p-1>0$, and $A$, such that
	
	\bea
	\mathcal{M}(\vp) \geq \ve(I-J)(\vp) - A.
	\eea
	
	The coercivity of $\mathcal{M}(\vp)$ guarantees that $J(\vp)$ is bounded above along the flow, hence $\int_X \vp \o^n_0$ is also bounded. The application of \ref{UniformSupremumBounds} finishes the rest of the proof. 
\end{proof}

\hfill

Using the previous proposition, we can establish $(8)$ and $(9)$ from \ref{TheoremOne}.

\hfill

\begin{proposition}\label{BoundOnInfAndOsc}
	If there is a uniform constant $C>0$, such that $\textrm{osc}(\vp) \leq C$, or $\inf_X \vp \geq C$, then the inverse Monge-Ampere flow converges to a twisted K\"ahler-Einstein metric, and $\textrm{osc}(\vp)$ is bounded. 
\end{proposition}

\hfill

\begin{proof}
	If $-\inf_X \vp \leq C$, then for any $p > 0$
	\bea
	{1 \over V} \int_X e^{-p\vp} \o^n_0
	\leq
	e^{-p\inf_X \vp},
	\eea
	and the statement follows from \ref{lpestimate}. The bound for $\textrm{osc}(\vp)$ follows from \ref{UniformSupremumBounds}.
\end{proof}

\hfill

Using the previous proposition, we can show that uniform bounds on $d_p(0,\vp)$ are equivalent to convergence of the flow.

\hfill

\begin{proposition}\label{MetricsBound}
	If for any $p \geq 1$ there is a uniform constant $C_p > 0$, such that $d_p(0,\vp) \leq C_p$ along the flow, then the inverse Monge-Ampere flow converges.
\end{proposition}

\hfill

\begin{proof}
	By \cite[Corollary 4]{DarvasGeometry}, there is a uniform constant $C(p)>0$, such that
	\bea
	\sup_X \vp \leq C(p)(d_p(0,\vp) + 1).
	\eea
	Hence, the bound on $d_p(0,\vp)$ implies the bound on $\sup_X \vp$. 
\end{proof}

\hfill

\begin{remark}\label{RemarkAboutD1}
	We know that along the flow 
	\bea
	J(\vp) -A \leq \sup_X \vp \leq J(\vp) + A.
	\eea
	
	By \cite[Theorem 6.2]{DarvasGeometry}, there is $B>1$, such that 
	\bea
	B^{-1}J(\vp) \leq d_1(0,\vp) \leq BJ(\vp).
	\eea
	Hence, for $p=1$ the statement of the proposition above follows directly from the properties of the flow. In fact, it shows that the flow converges if and only if $d_1(0,\vp)$ is bounded.
\end{remark}

\hfill

We finish this section by collecting all the estimates in one theorem, and prove \ref{TheoremOne}.

\hfill

\begin{theorem}(=\ref{TheoremOne})
	Assume that $(X,\o_0)$ is a compact K\"ahler manifold. Let $\b \geq 0$ be a semipositive $(1,1)$-form, such that $c_1(X) = [\o_0] + [\b]$. Assume further, that $X$ does not have holomorphic vector fields. 
	Then the following estimates along the inverse Monge-Ampere flow are equivalent and imply the convergence of the flow:
	\begin{itemize}
		\item[(1)] $\sup_X \vp$ is uniformly bounded from above;
		
		\item[(2)] The average ${1 \over V}\int_X \vp \o^n_0$ is uniformly bounded from above;
		
		\item[(3)] $J(\vp)$ and $d_1(0,\vp)$ are uniformly bounded from above;
		
		\item[(4)] $I(\vp)$ is uniformly bounded;
		
		\item[(5)] For some $p>1$ the integral ${1\over V}\int_X e^{-p\vp}\o^n_0$ is uniformly bounded.
	\end{itemize}
	
	The following bounds imply the convergence of the inverse Monge-Ampere flow.
	\begin{itemize}
		\item[(6)] The alpha-invariant $\a(X, \o_0)$ is bigger than ${n \over n+1}$;
		
		\item[(7)] $\textrm{osc}(\vp)$ is uniformly bounded along the flow;
		
		\item[(8)] $\inf_X\vp$ is uniformly bounded below along the flow;
		
		\item[(9)] For some $p > 1$ the distance $d_p(0, \vp)$ is uniformly bounded from above along the flow.
	\end{itemize}
	In particular, if $(X,\o_0)$ does not admit a twisted K\"ahler-Einstein metric, then both $||\vp||_{C^0}$ and ${1 \over V}\int_X \vp \o^n_0$ are unbounded.
\end{theorem}
\hfill

\begin{proof}
	The statement follows from \ref{SupBoundsForAlphaInv}, \ref{UniformSupremumBounds}, \ref{lpestimate}, \ref{BoundOnInfAndOsc}, and \ref{MetricsBound}. 
\end{proof}

\hfill

The following corollary is trivial, but it is worth to mention as a sort of counterpart to \cite[Proposition 2.7]{CHT17}.

\hfill

\begin{corollary}
	The inverse Monge-Ampere flow converges if there is a uniform constant $A>1$, such that the positively curved metrics $H(t):=e^\rho \o^n_\vp$ on $-K_X$ satisfy the following inequality:
	\bea
	A^{-1} H(0) \leq H(t) \leq AH(0).
	\eea
\end{corollary}

\hfill

\begin{proof}
	Note that $H^{-1}(0)H(t) = e^{-\vp + c(t)}$. Hence, the uniform bound on $-\vp + c(t)$ implies the bound on $\sup_X \vp$. 
\end{proof}

\section{Nadel multiplier ideal sheaves along the inverse Monge-Ampere flow.}\label{SectionNadel}

Let $(X,\o_0)$ be a Fano manifold, $\o_0 \in c_1(X)$ and $\vp=\vp(t)$ evolves along the inverse $MA^{-1}$ flow. In this section we show, that if $X$ does not admit a KE metric in $c_1(X)$, then the inverse Monge-Ampere flow produces a Nadel multiplier ideal sheaf.

First of all, we are going to recall the definition of a multiplier ideal sheaf (see \cite[Theorem and Definition 4.1]{DemaillyKollar2001}):

\hfill

\begin{definition}\label{MultiplierSheafDefinition}
	If $\vp$ is a psh function on a complex manifold $X$, the multiplier ideal sheaf $\cI(\vp) \subset \cO_X$ is defined by
	\bea
	\Gamma(U, \cI(\vp)) = \{f \in \cO_X(U) ~ \vline ~ |f|^2 e^{-2\vp} \in L^1_{loc}(U) \}
	\eea
	for every open set $U \subset X$. Then $\cI(\vp)$ is a coherent ideal sheaf in $\cO_X$. 
\end{definition}

\hfill

Next, we will need the Nadel vanishing theorem from \cite{Nadel1990}. The version here is a version of Nadel vanishing, formulated in \cite[Theorem 6.5]{DemaillyKollar2001}.

\hfill

\begin{theorem}\label{NadelVanishing}
	Let $(X, \o)$ be a Kähler orbifold and let $L$ be a	holomorphic orbifold line bundle over $X$ equipped with a singular hermitian metric $H = H_0 e^{-\vp}$ of weight $\vp$ with respect to a smooth metric $H_0$. Assume that the curvature $F_H$ is positive definite in the sense of currents, i.e. for some $\ve>0$, such that $F_H = -\sqrt{-1}\p\bar{\p}\log H \geq \ve \o$ as currents. If
	$K_X \otimes L$ is an invertible sheaf on X, we have
	\bea
	H^q(X,K_X \otimes L \otimes \cI(\vp)) = 0, ~~ \forall q \geq 1.
	\eea
\end{theorem}

\hfill

Finally, we recall the following existence criterion for K\"ahler-Einstein metrics on Fano manifolds from \cite[Theorem 0.1]{Nadel1990} in the form, stated in \cite[Theorem 6.4]{DemaillyKollar2001}.

\hfill

\begin{theorem}\label{ExistenceCriterionForKE}
	Let $X$ be a Fano orbifold of dimension $n$. Let $G$ be a compact subgroup of the group of complex automorphisms of $X$. Then $X$ admits a $G$-invariant Kähler–Einstein metric, unless $-K_X$ possesses a $G$-invariant singular
	hermitian metric $H = H_0 e^{-\vp}$ ($H_0$ being a smooth $G$-invariant metric and $\vp$ is a $G$-invariant
	function in $L^1_{loc}(X))$, such that the following properties occur:
	\begin{enumerate}
		\item The curvature current $F_H= -\sqrt{-1}\p \bar{\p}\log H$ of $H$ is semipositive, i.e. 
		
		\bea
		F_H= -\sqrt{-1}\p \bar{\p}\log H = -\sqrt{-1}\p \bar{\p}\log H_0 +  -\sqrt{-1}\p \bar{\p}\vp \geq 0;
		\eea
		
		\item For every $\a > {n \over n+1}$, the multiplier ideal sheaf $\cI(\a\vp)$ is nontrivial, (i.e. $0 \neq = \cI(\a\vp) \neq
		\cO_X$ ).
	\end{enumerate}
\end{theorem}

\hfill

Now we are in position to formulate the main result of this section.

\hfill

\begin{theorem}\label{NadelForInverseMA}
	(=\ref{TheoremTwo})
	Let $(X,\o_0)$ be a compact K\"ahler manifold, that satisfies all the assumptions of \ref{TheoremOne}, and does not admit a twisted K\"ahler-Einstein metric. Let $\a >{n \over n+1}$. Then there is a sequence of times $\{t_j\}$, and a sequence of solutions $\{\vp_{t_j}\}$ to the inverse Monge-Ampere flow, such that $\psi_j:= \vp_j - {1 \over V}\int_X \vp_j \o^n_{\vp_j}$ converge in $L^1(X, \o^n_0)$ to $\psi_\i$, and $\cI(\a\psi_\i)$ is a proper multiplier ideal sheaf, and

\bea
H^q(X,-\lfloor \a \rfloor K_X \otimes \cI(\a\psi_\i)) = 0, ~~ \forall q\geq 1.
\eea
\end{theorem}

\begin{proof}
	The proof goes along the lines to the proof of the similar statement for the K\"ahler-Ricci flow in \cite{Rub2009}.
	
	Let $\vp$ be a solution of the inverse Monge-Ampere flow. If $X$ does not admit the KE metric in $c_1(X)$, then \ref{TheoremOne} guarantees that $||\vp||_{C^0(X)}$ and ${1 \over V}\int_X \vp \o^n_0$ are unbounded along the flow. Hence, the results from Section \ref{SectionAlphaInvariant} (see \ref{SupBoundsForAlphaInv}), guarantee that there is a sequence of times $t_j$, such that for $\vp_j:=\vp_{t_j}$ the following holds:
	
	\bea
	\lim_{j \to +\i}\int_X e^{-({n \over n+1} + {1 \over j})(\vp_j - {1 \over V}\int_X \vp_j \o^n_0))} \o^n_0=+\i.
	\eea
	
	Furthermore, the same is true for any $\a > {n/n+1}$ (maybe after taking a subsequence):
	
	\bea
	\lim_{j \to +\i}\int_X e^{-\a(\vp_j - {1 \over V}\int_X \vp_j \o^n_0))} \o^n_0=+\i.
	\eea
	
	The set of positive currents in a given cohomology class is compact it the weak topology, so we can find yet another subsequence (which we still denote by $\{t_j\}$), such that the functions $\psi_j:=\vp_j - {1 \over V}\int_X \vp_j \o^n_0 $ converge in $L^1(X, \o^n_0)$ (see \cite[Theorem 6.4, pp.549-550]{DemaillyKollar2001} for more details). The effective version of semicontinuity theorem (see \cite[Theorem 0.2.2]{DemaillyKollar2001}, \cite{PS2000}), says, that
	
	\bea
	\int_X e^{-\a\psi_\i} \o^n_0=+\i, ~~\a>{n \over n+1}.
	\eea
	
	Thus, the multiplier ideal sheaf $\cI(\a\psi_\i)$ is a proper ideal sheaf of $\cO_X$. The statement about cohomology is a consequence of the Nadel vanishing theorem.
\end{proof}

\section{The $L^\i$-estimate for the inverse Monge-Ampere flow and its applications.}\label{SectionSuffering}

In this section we prove linear lower bound on $\inf_X \vp$ under the assumption on an exponential bound on the volume form $\o^n_\vp$ along the flow. This answers to the question in \cite[Section 5.2]{Xia2019}. Before we proceed, we need the following version of $C^2$-bound for the solution of the inverse Monge-Ampere flow.

\hfill

\begin{proposition}\label{C2estimate}
	There are constants $C>0$ and $A>1$, such that
	\begin{equation*}
		\log\Tr_{\o_0}\o_\vp \leq C+ A(\vp - \inf_X \vp) +t -\inf_X\vp 
	\end{equation*}
\end{proposition}

\hfill

\begin{proof}
	It was shown in \cite[Proposition 2.6]{CHT17}, that 
	\begin{equation*}
		\log\Tr_{\o_0}\o_\vp \leq C + A(\vp - \inf_X \vp) -\inf_X(\rho + \vp -c(t)).
	\end{equation*}
	However, 
	\begin{gather*}
		\inf_X(\rho + \vp -c(t)) \geq \inf_X \rho + \inf_X \vp - c(t) \geq\\ -t - A_2 + c(t) + \inf_X \vp - c(t) = -t + + \inf_X \vp
	\end{gather*}
	by \ref{KnownEstimates}.
\end{proof}

\hfill

How we are going to show the main result of this section.

\hfill

\begin{theorem}\label{InfEstimateAlongIMA}(=\ref{TheoremThree})
	Let $\vp$ be a solution to the inverse Monge-Ampere flow.
	Then there is a constant $M > 0$, such that 
	\bea
	||\vp||_{C^0} \leq M(t + 1)
	\eea
	along the inverse Monge-Ampere flow. 
\end{theorem}

\hfill

\begin{proof}
	By \ref{KnownEstimates}, we know that $\sup_X \vp \leq t + A_1$. So, we only need to show the upper bound for $-\inf_X \vp$. Also, without loss of generality, we may assume that $\textrm{E}(\vp) = 0$ along the inverse Monge-Ampere flow. This implies that $\sup_X \vp \geq 0$.
	
	In the following, $C$ denotes a constant, that may vary from line to line, but depends only on the initial data and independent on $t$.
	
	\textbf{Step 1:}
	Let $B= {A+1 \over 1 - \d}$, where $A>1$ is the same as in the statement of \ref{C2estimate}, and $\d>0$ will be chosen later. Let $\psi = \vp -\sup_X \vp$, and $u = e^{-B\psi}$.
	
	The Sobolev inequality says that for any smooth function $f$ on $X$ we have
	\bea
	\Bigg(\int_X f^{2n \over n-1} \o^n_0\Bigg)^{n-1 \over n}
	\leq
	C_S\Bigg(\int_X |\na f|^2 \o^n_0 + \int_X f^2 \o^n_0\Bigg),
	\eea
	and $C_S$ depends only on $(X,\o_0)$. If we take $f = u^{p \over 2}$ for $p \geq 1$, we can estimate the gradient term in the Sobolev inequality as follows:
	
	\begin{gather*}
		\int_X |\na e^{{-pB\psi \over 2}}|^2 \o^n_0
		=
		{Bpn \over 2}\int_X e^{-pB\psi}\sqrt{-1}\p \bar{\p}\psi \wedge \o^{n-1}_0
		\\ 
		=
		{Bpn \over 2}\int_X e^{-pB\psi}\big(\o_\vp - \o_0\big) \wedge \o^{n-1}_0
		\\
		\leq
		{Bp \over 2}\int_X e^{-pB\psi}\Tr_{\o_0}\o_\vp \o^n_0
		\\
		\leq
		{CBp \over 2}\int_X e^{-pB\psi + A(\vp - \inf_X \vp) + t - \inf_X\vp}  \o^n_0.
	\end{gather*}
	
	Here, in the last inequality we used \ref{C2estimate}.
	
	Adding, and subtracting $\sup_X \vp$, we obtain the following:
	
	\begin{equation}
		A(\vp - \inf_X \vp) + t - \inf_X\vp \leq A(\psi - \inf_X \psi) - \inf_X\psi + t + C \leq -(A+1)\inf_X \psi + t + C.
	\end{equation}
	Here $C$ is a constant, and the last inequality follows from $A\psi = A(\vp - \sup_X\vp) \leq 0$. Hence,
	
	\begin{equation}
		\int_X |\na e^{{-pB\psi \over 2}}|^2 \o^n_0 
		\leq 
		C_0pe^t e^{-(A+1)\inf_X \psi}\int_X e^{-pB\psi}  \o^n_0
		\leq C_0pe^t||u||^{1 - \d}||u||^p_p.
	\end{equation}
	
	Thus, if we set $\tau = {n \over n-1}$, then we obtain the following
	
	\begin{equation}\label{AlmostWeinkove}
		||u||_{p\tau} \leq (C_0p)^{1/p}e^{t/p}||u||^{{1 - \d \over p}}_{C^0}||u||_p.
	\end{equation}
	
\textbf{Step 2:}

Set $p_k = p\tau^k$ and $a(k) = 1 + {1 \over \tau} + \dots + {1 \over \tau^k}$. By iteration, we get
\begin{equation}
	||u||_{p_{k+1}} \leq C(k)e^{ta(k)/p}||u||^{(1 - \d)a(k)/p}_{C^0}||u||_p,
\end{equation}
where $C(k) = (Cp)^{a(k)/p}$. As $k$ tends to infinity, the limit $\lim_{k \to +\i}a(k) = {1 \over 1 - 1/\tau} = n$ is finite. Taking the limit, we obtain
\begin{equation}
	||u||_{C^0} \leq C_1 e^{tn/p}||u||^{(1 - \d)n/p}_{C^0}||u||_p.
\end{equation}

If we take $p=1$ and take $\d$ such that $\gamma = (1 - \d)n <1$, then 
\begin{equation}
	||u||^{1-\gamma}_{C^0} \leq C_1 e^{tn}||u||_1.
\end{equation}

We cannot bound $||u||_1$ by $\a$-invariant anymore, as it is done in \cite{W0306}. Instead, we invoke the Trudinger inequality from \cite[Theorem 7]{GPT23}. The Trudinger inequality implies that
\begin{equation}
	\log||u||_1 = \log\Bigg({1 \over V}\int_X e^{-B\psi}\o^n_0\Bigg)\leq {C \over V}\int_X (-\psi)\o^n_\psi + C
\end{equation}
for some uniform constant $C$, which depends only on $B$ and the initial data. Since $\psi = \vp - \sup_X\vp$, and $\vp$ is a solution to the inverse Monge-Ampere flow, \ref{SupBounds} together with \ref{KnownEstimates} imply that
\begin{equation}
	{1 \over V}\int_X (-\psi)\o^n_\psi \leq (n+1)\sup_X\vp + C \leq (n+1)t + C.
\end{equation}
Hence, there are $C_2,C_3>0$, such that $||u||^2_2 \leq e^{C_2(t+1)}$, and
\begin{equation}\label{TheEstimate}
	||u||_{C^0} \leq C^{{1 \over 1-\gamma}}_1e^{C_3(t+1)}.
\end{equation}

\textbf{Step 3:}

Taking the logarithm in inequality \ref{TheEstimate}, we obtain the following:
\bea
B(\sup_X \vp - \inf\vp) \leq C_3(t+1) + {1 \over 1- \gamma }\log C_1. 
\eea

By \ref{KnownEstimates}, we know that $t+ A_1 \geq \sup_X \vp \geq \dot{c}(0)t + c(0) + \inf_X \rho_0$, and $-\dot{c}(0) = {1 \over V}\int_X (e^\rho_0 - 1)^2 \o^n_0 \geq 0$. Hence 

\bea
-\inf_X \vp \leq M(t+1).
\eea

\end{proof}

\hfill

An immediate consequence of this theorem is a local bound of asymptotic geodesic rays, constructed in \cite{Xia2019} and \cite{His2019}. We briefly recall the idea. We use the notation of \cite{His2019}, and denote by $\vp_t$ the solution to the inverse Monge-Ampere flow.

Pick sequence of times $t_j \to \i$, and suppose that $\lim_{j \to +\i}\sup_X \vp_{t_j} = +\i$. We construct the geodesic segments $\vp^{t_j}_j$, which connect $\vp_0$ with $\vp_{t_j}$, and parametrize them by the time $t \in [0,t_j]$. By \cite[Theorem 3.4]{DarvasGeodRays}, together with \ref{TheoremThree}, the geodesic segments satisfy the following bounds
\bea
-C \leq\sup_X {\vp^a_j - \vp^b_j \over a - b} = \sup_X{\vp^{t_j}_j - \vp^0_j \over t_j} =  \sup_X{\vp_{t_j} - \vp_0 \over t_j}\leq C
\\
-C \leq \inf_X {\vp^a_j - \vp^b_j \over a - b} = \inf_X {\vp^{t_j}_j - \vp^0_j \over t_j} = \inf_X {\vp_{t_j} - \vp_0 \over t_j}\leq C
\eea

for some uniform constant $C>0$. 

Without loss of generality, we may assume that $\textrm{E}(\vp_t)= 0$, and $\vp_0 = 0$. Thus, by the discussion above and \cite[Theorem 1]{DarvasGeodRays}, we get

\bea
-C \leq\sup_X{\vp^{t}_j \over t} \leq C
\\
-C \leq  \inf_X {\vp^{t}_j  \over t} \leq C.
\eea

By \cite[Section 3.2]{His2019}, we know that for any finite time $T>0$ the entropy ${1 \over V}\int_X \log \Big({\o^n_\vp \over \o^n_0}\Big)\o^n_\vp$ is bounded. After a possible reparametrization of $\vp^t_j$ by its $d_2$ arc length, we can use \cite[Theorem 2.4]{DarvasHe}, and construct the geodesic ray $\vp^t$, as in  \cite[Theorem 3.2]{DarvasHe}, which is asymptotic to the flow in a sense that for any $p \geq 1$ we have $\lim_{j \to +\i}d_p(\vp^t_j, \vp^t) = 0$. Because $\vp^t$ is a bounded geodesic ray, we can use \cite[Proposition 2.1]{DarvasHe} to deduce, that $\vp^t$ is a geodesic for all $d_p$.

We established the first part of \ref{TheoremFour}, which is similar to \cite[Theorem 3.2 and Theorem 3.3]{DarvasHe} in the case of K\"ahler-Ricci flow.

\hfill

\begin{theorem}
	Assume that $\vp_t$ is a diverging trajectory of the inverse Monge-Ampere flow, which satisfies the assumptions of \ref{TheoremThree}. Then there exist a curve $\vp^t$, such that for any $p \geq 1$ it is nontrivial $d_p$ geodesic ray, weakly asymptotic to $\vp_t$. Moreover:
	\begin{itemize}
		\item[(1)] The functional $\mathcal{F}$ is convex and decreasing along the geodesic ray;
		
		\item[(2)] The normalized ray $\vp^t - \sup_X (\vp^t - \vp^0)$ converges to  $\vp^\i$, such that ${1 \over V}\int_X e^{-{n \over n+1}\vp^\i}\o^n_0 = +\i$.
	\end{itemize}
\end{theorem}

\hfill

\begin{proof}
	We already established the existence of asymptotic geodesic ray. Thus, we need to prove only (1) and (2).
	
	\begin{itemize}
		\item[(1)] Since $\lim_{j \to +\i}d_2(\vp^t_j,\vp^t) =0$, and $\mathcal{F}$ is continious with respect to $d_2$, the proof of this claim is the same, as \cite[Theorem 3.3]{DarvasHe};
		
		\item[(2)] By \ref{SupBoundsForAlphaInv}, and the property of $d_1$-distance along the inverse Monge-Ampere flow, we have
		\begin{gather}\label{DistanceAndAlpha}
			-\log \Bigg({1 \over V}\int_X e^{-\a(\vp_t - \sup_X \vp_t) + \rho_0} \o^n_0\Bigg)
			\leq
			\\
			-\log \Bigg({1 \over V}\int_X e^{-\a(\vp_t - \sup_X \vp_t)} \o^n_0\Bigg) + C_1
			\leq
			-((n+1)\a - n)Cd_1(0,\vp_t) + C_2.
		\end{gather}
		Now we pick $\a \in \big({n \over n+1};1\big)$.	Without loss of generality, we again assume that $\vp_0 = 0$, and $\textrm{E}(\vp_t) = 0$. Consider the geodesic segments $\vp^t_j$ as above, and assume it is naturally parametrized. The same reasoning, as in \cite[Theorem 3.2, part 2]{DarvasHe} shows, that 
		
		\begin{gather*}
			-\log \Bigg({1 \over V}\int_X e^{-\a(\vp^t_j - \sup_X \vp^t_j) + \rho_0} \o^n_0\Bigg) =
			-\log \Bigg({1 \over V}\int_X e^{-\a\vp^t_j + \rho_0} \o^n_0\Bigg) - \a \sup_X \vp^t_j
		\end{gather*}
	is convex as well. By \ref{DistanceAndAlpha}, we know, that
	\bea\label{Geee}
	-\log \Bigg({1 \over V}\int_X e^{-\a(\vp^t_j - \sup_X \vp^t_j) + \rho_0} \o^n_0\Bigg) \leq - \ve d_1(0,\vp^t_{j}) + C
	\eea
	for all $t$. Indeed, if $d_2(0,\vp_{t_j}) = t$, then $d_1(0,\vp_{t_j}) = \ve_j t$ for some $\ve_j > 0$, and the claim follows from the fact, that for any $\a \in \big({n \over n+1}; 1\big)$ the integral in \ref{Geee} is a convex function of $t$ (see \cite{Brn1,DarvasHe}). Since $\lim_{j \to +\i}d_p(\vp^t_j, \vp^t) = 0$, and the integral on the left hand side of \ref{Geee} is continuous, we obtain
	\bea
	-\log \Bigg({1 \over V}\int_X e^{-\a(\vp^t - \sup_X \vp^t) + \rho_0} \o^n_0\Bigg) \leq - \ve d_1(0,\vp^t) + C.
	\eea
	As in \cite[Introduction]{DarvasGeodRays} and \cite[Theorem 2.6]{DarvasHe}, the function $\vp^t - \sup_X \vp^t$ is decreasing in $t$ due to our choice of normalization. Hence, there exists $\vp^\i = \lim_{t \to +\i}(\vp^t - \sup_X \vp^t)$, which is not identically $-\i$. Again, as in \cite{DarvasHe}, the proof of the openness conjecture (see \cite{GZhou, Brn2}) implies that ${1 \over V}\int_X e^{-{n \over n+1}\vp^\i}\o^n_0 = +\i$.
	\end{itemize}
	
\end{proof}

\hfill

\noindent Department of Mathematics, Columbia University, New York, NY 10027 USA

\noindent nk2957@columbia.edu

\noindent nklemyatin@math.columbia.edu

\end{document}